\documentclass[10pt,a4paper]{article}
\usepackage[utf8]{inputenc}
\usepackage{amsmath}
\usepackage{amsfonts}
\usepackage{amssymb}
\usepackage{mathtools}
\usepackage{amsthm}
\usepackage{setspace}
\usepackage{yhmath}
\usepackage[left=1in,right=1in,top=1in,bottom=1in]{geometry}
\usepackage{enumitem}
\usepackage{tikz-cd}
\usepackage{tikz}
\usetikzlibrary{shapes.geometric, arrows}

\usepackage{hyperref}
\usepackage[symbol]{footmisc}

\newtheorem{theorem}{Theorem}[section]
\newtheorem{lemma}[theorem]{Lemma}

\newtheorem{proposition}[theorem]{Proposition}

\newtheorem{corollary_p}[theorem]{Corollary}

\newtheorem{definition}[theorem]{Definition}
\newtheorem*{remark}{Remark}

\theoremstyle{definition}
 
\usepackage{algorithmic}

\tikzstyle{startstop} = [rectangle, rounded corners, 
minimum width=3cm, 
minimum height=1cm,
text centered, 
draw=black]
\tikzstyle{starts} = [rectangle, rounded corners, 
minimum width=3cm, 
minimum height=1cm,
text centered,
text width=2cm,
draw=black]
\tikzstyle{process} = [rectangle, rounded corners,
minimum width=3cm, 
minimum height=1cm, 
text centered, 
text width=5cm, 
draw=black]
\tikzstyle{arrow} = [thick,->,>=stealth]

\title{Intrinsic Stochastic Differential Equations and Extended Ito Formula on Manifolds}

\author{Sumit Suthar, Soumyendu Raha\\ Indian Institute of Science, CV Raman Road, Bengaluru- 560012}

\date{July 27, 2023}

\begin{document}
\maketitle
\begin{abstract}
A general way of representing Stochastic Differential Equations (SDEs) on smooth manifold is based on Schwartz morphism. In this manuscript we are interested in SDEs on a smooth manifold $M$ that are driven by p-dimensional Wiener process $W_t \in \mathbb{R}^p$. In terms of Schwartz morphism, such SDEs are represented by Schwartz morphism that morphs the semi-martingale $(t,W_t)\in\mathbb{R}^{p+1}$ into a semi-martingale on the manifold $M$. We show that it is possible to construct such Schwartz morphisms using special maps that we call as diffusion generators. We show that one of the ways of constructing diffusion generator is by using regular Lagrangian. Using this diffusion generator approach, we also give extended Ito formula (also known as generalized Ito formula or Ito-Wentzell's formula) for SDEs on manifold.
\end{abstract}
\section{Introduction}
Stochastic Differential Equation (SDE) evolving on linear spaces is a well studied subject. Some of the popular books on this subject are \cite{arnold1974stochastic,oksendal2013stochastic}. On manifolds, however, the subject of SDEs is an active research area. Ever since K. Ito first described the coordinate transformation rules on manifolds, the subject has evolved and taken a form of what is now broadly known as Stochastic Differential Geometry. In linear spaces, Stratonovich SDE representation and Ito SDE representation are two popular ways of representing semi-martingale in form of SDEs. It is natural that there will be equivalent ways of describing SDEs on manifold. In case of Stratonovich SDEs, it is enough to consider sections of tangent bundle (vector fields) to describe the drift and the noise coefficients. However, similar statement cannot be made for Ito type SDE due to the additional drift correction term. To address this problem L. Schwartz, in \cite{schwartz1982geometrie}, introduced the idea of the \textit{second order tangent bundle}. It is because of this special construction that the study of Stochastic Differential Equations on manifolds gets a special name of Stochastic Differential Geometry. A complete account of Schwartz's second order geometry can be found in \cite{emery2012stochastic}. One of the central ideas in Schwartz's Stochastic Differential Geometry is that the stochastic differential is considered as an infinitesimal element of Schwartz's second order tangent space. These stochastic differentials are also called Intrinsic differentials or Schwartz differentials.

In the book \cite{gliklikh2011global}, Ito SDEs on manifolds are developed using the idea of Ito-bundle. As per this approach, if a manifold is equipped with a connection, then it is possible to describe an Ito SDE on the manifold as a section of the Ito-bundle. The key highlight of the book is the description of Ito Equation on manifolds using the Belopolskya-Daletskii form (section 7.3 of \cite{gliklikh2011global}), which can be exploited for numerical computations. Yet another approach is that of Stochastic development and anti-development, that can be found in chapter 2 of \cite{hsu2002stochastic} or in \cite{elworthy1982stochastic}. However, we do not consider this approach here.

In the Schwartz's approach or the so called \textit{Intrinsic SDEs}, the description of Intrinsic SDEs depend on Schwartz morphism that can morph semi-martingales from a source manifold to a semi-martingale on the target manifold. If we consider the source manifold as $\mathbb{R}^{p}$ with $(X_t)$ as a semi-martingale on $\mathbb{R}^{p}$, then the Schwartz morphism will morph the semi-martingale $X_t\in \mathbb{R}^{p}$ into a semi-martingale on some target manifold, $M$. As per Schwartz's approach, if there is a map $F:\mathbb{R}^{p}\to M$ then it is possible to obtain Schwartz morphism that morphs the semi-martingale $Y_t\in \mathbb{R}^{p}$ into the semi-martingale $F(Y_t)\in M$. However, very often we do not have the function $F:\mathbb{R}^{p}\to M$.  Hence, the problem remains in the construction of the Schwartz morphism. In this article we focus on constructing Schwartz morphisms that morphs the process $(t,W_t)\in \mathbb{R}^{p+1}$ into a semi-martingale on $M$. This has been considered as example in chapter 1 of \cite{rossi2022rough}, wherein the author has demonstrated that both Ito-bundle formulation of SDE and Stratonovich SDE can be reformulated as Schwartz's Intrinsic SDE. Conversion formulae between Ito-Stratonovich-Intrinsic SDEs are also given.
An approach based on 2-jets is presented in \cite{armstrong2018intrinsic}. This approach can be interpreted as construction of the Schwartz morphism using 2-jets of $F:\mathbb{R}^p \to M$. The authors have focused on construction of iterative schemes that can be used for numerical computation of the solution.

If $F:\mathbb{R}\times M\to N$, such that $F(t,x)$ is a semi-martingale for every $x\in M$, and $X_t$ is a semi-martingale on $M$; then the SDE representation for the semi-martingale $F(t,X_t)$ is not a straight forward application of Schwartz morphism. In Euclidean spaces the Ito SDE representation for $F(t,X_t)$ is given by the extended Ito formula \cite{kunita1981some}. We are interested in deriving \textit{extended Ito type formulae} on manifolds. From an application perspective Ito formula is used for finding stochastic characteristic of a Partial Differential Equation (PDE) \cite{constantin2008stochastic, holm2020Ito_wentzell}. 

In this article, we consider an alternate view point to describe Intrinsic SDEs that may also be obtained using Schwartz morphism that morphs the process $(t,W_t)\in \mathbb{R}^{p+1}$ into a semi-martingale on $M$. This alternate view point is based on what we call as \textit{diffusion generators}. We show that the diffusion generator can be used to construct Schwartz morphism that morphs the process $(t,W_t)\in \mathbb{R}^{p+1}$ into a semi-martingale on $M$ (section \ref{section:diffusion_generator_Schwartz morphism}). We show that diffusion generator can be constructed using regular Lagrangian (section \ref{section:construction using Lagrangians}). Finally, we use the diffusion generator approach to derive {extended Ito formulae} on manifolds (section \ref{section:last}). Our contributions from this article are discussed in section \ref{section:contributions}.

\subsection{Basic definitions and notations.}
We will denote the set of all sections of any fiber bundle $F$ by $\Gamma(F)$. The set of all smooth vector fields will be denoted by $\mathfrak{X}(M)$ and the set of all smooth function by $\mathfrak{F}(M)$. {\textit{Schwartz's second order tangent space}} at a point $x$ on an n-manifold $M$ is defined as a vector space of all differential operators of upto order 2 at point $x$.  We will denote it as $\mathfrak{D}_xM$. Locally, every second order differential operator is symmetric and is represented as $\partial^2_{ij}$. Therefore, every differential operator upto second order is locally of the form  $a^i\partial_i + b^{ij}\partial^2_{ij}$. Symmetry of the second order differential operators means that the dimension of the second order tangent space is $n + (1/2)n(n+1)$. We will call the elements of \textit{Schwartz's second order tangent space} $\mathfrak{D}_xM$ as \textbf{\textit{diffusors}} at point $x\in M$. With these definitions, it is clear that a tangent vector is also a diffusor i.e. $T_xM\subset\mathfrak{D}_xM$ $\forall$ $x\in M$.\\
For any manifolds $M$ and $N$, consider $L\in \mathfrak{D}_xM$; if $\phi:M\to N$, then the push forward of $L$ by $\phi$ at a specific point $x\in M$ is written as $\mathfrak{D}_x\phi (L)$ such that $\mathfrak{D}_x\phi:\mathfrak{D}_xM\to \mathfrak{D}_{\phi(x)}N$. Moreover, $\forall f\in \mathfrak{F}(N)$, $\mathfrak{D}_x\phi (L) [f] = L[f(\phi)] = L[\phi^*f]$. This push-forward map is linear. The fiber bundle over the manifold $M$, with Schwartz's second order tangent space $\mathfrak{D}_xM$ as the fibers, is called \textit{Schwartz's second order tangent bundle}. For brevity, we will call Schwartz's second order tangent bundle as \textbf{\textit{diffusion bundle}}, and Schwartz's second order tangent space as \textbf{\textit{diffusion space}}. A smooth \textbf{\textit{diffusor field}} $\zeta$ is defined as a smooth section of the diffusion bundle $\mathfrak{D}M$. Following our usual symbol for section of a fiber bundle, the set of all smooth diffusor field will be denoted by $\Gamma(\mathfrak{D}M)$. For $\phi:M\to N$, we will call the fiber preserving map over $\phi$, $\mathfrak{D}\phi:\mathfrak{D}M\to \mathfrak{D}N$ as the \textit{\textbf{diffusion map}}. Locally in charts $(U,\Upsilon)$ on $M$ and $(V,\chi)$ on $N$, for all $L\in \mathfrak{D}M$ such that $L|_U = a^i\partial_i + b^{ij}\partial^2_{ij}$,
\begin{equation}
\mathfrak{D}\phi \left(L\right)|_V = \left[ a^i\partial_i\phi^k + b^{ij}\partial^2_{ij}\phi^k\right] \partial_k + \left[b^{ij}\partial_i\phi^k \partial_j\phi^l\right]\partial^2_{kl}.
\end{equation}
\noindent Given $L\in \mathfrak{D}M$, consider a symmetric contravariant tensor $\hat{L}\in T^2_0M$ such that
\begin{equation}\hat{L}(df, dg) = \dfrac{1}{2}(L[fg] - fL[g] - gL[f]).\end{equation} The fact that $\hat{L}$ is indeed symmetric can be verified locally by considering $L = a^i\partial_i + b^{ij}\partial^2_{ij}$. So, locally \begin{equation}\hat{L}(df,dg) = b^{ij}\partial_if\partial_jg.\end{equation}
In other words, $\hat{L}$ can be interpreted as the symmetric part of the diffusor $L$.\\
\noindent A stochastic process $X_t$ on a manifold $M$ is said to be a \textbf{\textit{semimartingale}} if $f(X_t)$ is a semi-martingale $\forall$ $f\in \mathfrak{F}(M)$. Let $X_t$ be a continuous semi-martingale on manifold $M$. If $X_t^i$ are the local components of $X_t$ in some chart, then the local Ito differentials $dX_t^i$ and $\dfrac{1}{2}d[X_t^i,X_t^j]$ can be taken as coefficients to construct an infinitesimal diffusor
\begin{equation}
\textbf{d}X_t = (dX^i_t)\partial_i + \left(\dfrac{1}{2}d[X_t^i,X_t^j]\right)\partial^2_{ij}.
\end{equation}
The diffusor $\textbf{d}X_t$ is known as the \textbf{\textit{Schwartz differential}} of $X_t$.\\
If there are two manifolds $M$ and $N$ with $x\in M$ and $y\in N$ and there exists a linear map $J(x,y):\mathfrak{D}_xM\to \mathfrak{D}_yN$ such that $Img(J|_{T_xM})\subset T_yN$ and $\widehat{JL} = (J|_{T_xM}\otimes J|_{T_xM}) \hat{L},$ then such a map $J$ is called a \textbf{\textit{Schwartz morphism}}.\\
As per Schwartz's stochastic differential geometric approach, a Stochastic Differential Equation (SDE) for a process $X_t$ on a manifold $M$ is defined as
\begin{equation}
\mathbf{d}X_t = J(Y_t,X_t)\mathbf{d}Y_t,
\end{equation}
where $J$ is a Schwartz morphism from manifold $N$ to manifold $M$, and $Y_t$ is a given semi-martingale on the manifold $N$. This equation is known as \textit{Schwartz's SDE} or \textbf{Intrinsic SDE}.

\subsection{Background}
\label{section:pre}
In order to represent a semi-martingale on $M$ as an Intrinsic SDE, we need a semi-martingale on some manifold $N$ and a Schwartz morphism from manifold $N$ to $M$. If we consider $N$ to be a Euclidean space with $Y_t$ as a semi-martingale on $N$, then the problem remains in finding the Schwartz morphism from $N$ to $M$. The following well know theorem states that if we have a smooth map $\phi:N\to M$, then the Schwartz morphism from $N$ to $M$ is given by the diffusion map $\mathfrak{D}\phi$. Readers can refer \cite{emery2012stochastic} for the proof of the theorem.
\begin{theorem}[\cite{emery2012stochastic}]
If $\phi:N\to M$ is a smooth map, then the diffusion map $\mathfrak{D}_{x}\phi:\mathfrak{D}_xN\to \mathfrak{D}_{\phi(x)}M$ is a Schwartz operator from point $x$ $\to$ $\phi(x)$. Moreover, if $U_t$ is a semi-martingale on $N$, then the semi-martingale $\phi(U_t)$ on $M$ is given by the solution of the Schwartz Stochastic Differential Equation (SDE),
\begin{equation}
\textbf{d}X_t = \mathfrak{D}_{U_t}\phi(\textbf{d}U_t).
\end{equation}
In other words, the Schwartz differential $\textbf{d}(\phi(U_t))$ is obtained by the push forward of the Schwartz differential $\textbf{d}U_t$ by $\phi$; i.e. $\textbf{d}(\phi(U_t)) = \mathfrak{D}_{U_t}\phi(\textbf{d}U_t)$.
\end{theorem}

Although this theorem gives us a Schwartz morphism from $N$ to $M$, it depends on the map $\phi:N\to M$. The map $\phi:N\to M$ may not be available. In such cases, construction of Schwartz morphism becomes problematic. As per the following theorem from \cite{emery2012stochastic}, the Schwartz morphism can be constructed using the flow of differential equation that is defined using the linear map $S(y,x):T_yN\to T_xM$. The operator $S$ is known as \textit{Stratonovich operator}. Readers can refer to \cite{emery2012stochastic} for the proof of the following theorem.
\begin{theorem}[\cite{emery2012stochastic}]
\label{theorem:stratonovich_to_Schwartz}
For every Stratonovich operator $S(y,x):T_yN\to T_xM$, there exists a unique Schwartz operator $J(y,x):\mathfrak{D}_yN\to \mathfrak{D}_xM$, such that the Stratonovich SDE $\delta X_t = S(U_t,X_t)\circ \delta U_t$ has the same solution as that of the Intrinsic SDE $\mathbf{d}X_t = J(U_t,X_t)\mathbf{d}U_t$; such that, for smooth curves $(x(t),y(t))\in M\times N$, if $ \dot{x}(t) = S(y(t),x(t))\dot{y}(t)$, then $\dfrac{\mathbf{d}x(t)}{dt} = J(y(t),x(t)) \dfrac{\mathbf{d}y(t)}{dt}$.
\end{theorem}
Stratonovich SDEs of type
\begin{equation}
\delta X_t = V(X_t) dt + \sum_{l = 1}^p \sigma_l(X_t)\circ dW^l_t,
\end{equation}
where $V, \sigma_1, \sigma_2, ..., \sigma_p \in \mathfrak{X}(M)$; can be written in terms of Stratonovich operator $S$ from $\mathbb{R}^{p+1}$ to $M$. The Stratonovich operator that is given as $S(y,x):\mathbb{R}^{p+1}\to T_xM$ such that
\begin{equation}
S(y,x) = V(x)y_0 + \sum_{l = 1}^p \sigma_l(x) y^l,
\end{equation}
where $y = (y_0, y_1, y_2, ..., y_p)$. Now using theorem \ref{theorem:stratonovich_to_Schwartz}, we can obtain a Schwartz morphism $J$ to define the Intrinsic SDE as
\begin{equation}
\mathbf{d}X_t = J((t,W_t),X_t) \mathbf{d}(t,W_t),
\end{equation}
where $W_t = (W^1_t, ..., W^p_t)$ is a p-dimensional Wiener process. However, we wish to define Schwartz morphism without explicitly depending on the underlying Stratonovich morphism. Moreover, theorem \ref{theorem:stratonovich_to_Schwartz} does not tell us if there will be an underlying Stratonovich operator for every Schwartz morphism.

Let us consider an arbitrary Schwartz morphism $\beta$ from $\mathbb{R}^{p+1}$ to $M$. We know that, locally in chart $(U,\chi)$,
\[\beta(y,x) L|_U = \left( f^i_l(x) a^l + g^i_{lm}(x)b^{lm}\right) \partial_i + \left(f^i_l(x) f^j_m(x) b^{lm} \right) \partial^2_{ij},\]
for every $L \in \mathfrak{D}_y\mathbb{R}^{p+1}$ such that $L = a^l\partial_l + b^{lm}\partial^2_{lm}$ and the indices $l,m \in \{0,1, 2, ..., p\}$. Here $f^i_l, g^i_{lm}$ are local coefficients of $\beta$.  With this Schwartz morphism $\beta$, if we consider the SDE
\[\mathbf{d}X_t = \beta((t,W_t),X_t) \mathbf{d}(t,W_t),\]
then we find that
\begin{equation}
\label{equation: 2.intermediate}
\mathbf{d}X_t|_U =\left[ f^i_0(X_t)\partial_i +\dfrac{1}{2}\left( \sum_{l=1}^p g^i_{ll}(X_t)\partial_i +(f^i_l(X_t) f^j_l(X_t))\partial^2_{ij}\right)\right] dt + \sum_{l=1}^p(f^i_l(X_t) \partial_i)dW^l_t.
\end{equation}
Note that the term in the parenthesis is a diffusor. Therefore if we consider vector fields $V,\sigma_1, ..., \sigma_p\in \mathfrak{X}(M)$, and diffusor field $\alpha\in \Gamma(\mathfrak{D}M)$; then the following equation,
\begin{equation}
\label{equation: 2.intermediate_global}
\mathbf{d}X_t = Vdt +\dfrac{1}{2}\alpha dt + \sum_{l=1}^p\sigma_ldW^l_t,
\end{equation}
is a co-ordinate invariant representation of equation \eqref{equation: 2.intermediate} if the diffusor field $\alpha$ is such that 
\begin{equation}
\widehat{\alpha} = \sum_{l=1}^p\sigma_l \otimes \sigma_l.
\end{equation}
SDEs in form of equation \eqref{equation: 2.intermediate_global} (i.e. the SDEs on manifolds that are driven by Wiener process) have been well known and has been studied several times. The problem lies in obtaining the diffusor field $\alpha$. One of the ways of obtaining the diffusor field $\alpha$ in equation \eqref{equation: 2.intermediate_global} is by using theorem \ref{theorem:stratonovich_to_Schwartz}. Based on this theorem, if we consider Stratonovich differential equation
\[\partial X_t = V(X_t) dt + \sum_{l = 1}^p\sigma_l(X_t) \circ dW^l_t,\]
then it is easy to verify that the equivalent Intrinsic SDE is 
\[\mathbf{d}X_t = \left[V(X_t) + \dfrac{1}{2}\alpha_S(X_t) \right] dt + \sum_{l = 1}^p\sigma_l(X_t) \circ dW^l_t,\]
where $\alpha_S$ is locally given as
\[\alpha_S|_U = \sum_{l = 1}^pd\sigma^i_l\cdot\sigma_l\dfrac{\partial}{\partial x^i} + \sum_{l = 1}^p\sigma^i_l \sigma^j_l\dfrac{\partial^2}{\partial x^i \partial x^j}.\]
But, as already mentioned, we are not interested in Stratonovich SDEs.

From \cite{emery2006two}, we know that the following short-exact equation is valid fiber-wise on every point $x\in M$.
\[
\begin{tikzcd}
0 \arrow[r] & T_xM \arrow[r, "i"] & \mathfrak{D}_xM \arrow[r, "\widehat{\cdot}"] & T_xM\odot T_xM \arrow[r] & 0.
\end{tikzcd}
\]
This implies that there exists an isomorphism $J_x:\mathfrak{D}_xM\to T_xM \oplus (T_xM\odot T_xM)$. Moreover, if we represent a linear map from $\mathfrak{D}_xM$ to $T_xM$ as $Q_x$, then 
\begin{equation}\label{equation: J_connection}
J_x(\cdot) = (Q_x(\cdot), \widehat{\cdot}).
\end{equation}

In chapter 7 of \cite{emery2012stochastic}, it has been demonstrated that such linear maps $Q_x$ can be uniquely identified with a connection on the manifold. Therefore, from equation \eqref{equation: J_connection}, it is possible to construct the isomorphism $J_x$ using a connection on the manifold. Due to the isomorphism $J_x$, a diffusor field $\lambda$ can be obtained from vector field $V$ by considering
\[\lambda_x = J_x^{-1}((V_x,V_x\otimes V_x)),\] where $Q_x:\mathfrak{D}_xM\to T_xM$ is the linear map corresponding to the given connection. Moreover, if $\Gamma$ is the Christoffel symbol for the connection, the diffusor field $\lambda$ is locally given as 
\[\lambda|_U = -\Gamma^i_{jk}V^jV^k\partial_i + V^iV^j\partial_{ij}.\]

Therefore, we find that there exists a diffusor field equation $\alpha_I$ such that it is locally given as
\[\alpha_I|_U = \sum_{l=1}^p -\Gamma^i_{jk}\sigma^j_l\sigma^k_l\partial_i + \sum_{l=1}^p\sigma^i_l \sigma^j_l\partial_{ij}.\]
Since $\widehat{\alpha_I} = \sum_{l=1}^p\sigma_l \otimes \sigma_l$, the following SDE gives us a special case of equation \eqref{equation: 2.intermediate_global},
\begin{equation}
\label{equation: 2.intermediate_global_Ito}
\mathbf{d}X_t = Vdt +\dfrac{1}{2}\alpha_I dt + \sum_{l=1}^p\sigma_ldW^l_t.
\end{equation}
Such equations are called Ito SDEs on manifolds. In order to differentiate the Ito SDEs on manifolds from the Ito SDEs on Euclidean spaces, we will call the Ito SDEs on manifolds as \textit{standard Ito SDEs}. Standard Ito SDEs require a connection. This idea of using the notion of connection to construct the diffusor field $\alpha$ in equation \eqref{equation: 2.intermediate_global} was originally presented by Meyer in \cite{meyer1981geometrie} (in French). As English speakers, we find chapter 6 and chapter 7 of \cite{emery2012stochastic} as a useful reference. A modern approach that uses the idea of connections and Ito-bundle to derive equation \eqref{equation: 2.intermediate_global}, can be found in \cite{gliklikh2011global}.

Our interest in this article is in considering ways of constructing the diffusor field $\alpha$ in equation \eqref{equation: 2.intermediate_global}, without using the notion of connection and without depending on the underling Stratonovich morphism. To this end, we consider an alternate view point.

\bigskip

\noindent \textbf{Alternate view point of equation \eqref{equation: 2.intermediate_global}:}  If the diffusor field $\alpha$ is considered to be a sum of diffusor fields $\alpha_l$ (i.e. $\alpha = \sum_{l=1}^p \alpha_l$), such that for each $\alpha_l\in \Gamma(\mathfrak{D}M)$
\begin{equation}
\widehat{\alpha_l} = \sigma_l \otimes \sigma_l,
\end{equation}
then the equation \eqref{equation: 2.intermediate_global} converts into
\begin{equation}
\label{equation: 2.intermediate_global_alternate}
\mathbf{d}X_t = Vdt + \sum_{l=1}^p \left(\dfrac{1}{2}\alpha_l dt + \sigma_ldW^l_t\right).
\end{equation}
As each $\alpha_l$ has the property that $\widehat{\alpha_l} = \sigma_l \otimes \sigma_l$, each diffusor field $\alpha_l$ is associated with the vector field $\sigma_l$.

Due to this property of the diffusor field $\alpha_l$, which requires the noise vector field $\sigma_l$, it is natural to ask if we can construct diffusors from given vectors. In order to achieve this, we need a function that maps from tangent space $T_xM$ to the diffusion space $\mathfrak{D}_xM$. In other words, we need a fiber preserving map from $TM$ to $\mathfrak{D}M$ over identity.

Therefore, if we have a fiber preserving map $G: TM \to \mathfrak{D}M$ over identity, then the diffusor fields $\alpha_l$ can be obtained using the vector fields $\sigma_l$ as 
\[\alpha_l = G(\sigma_l).\]
As we have to ensure that $\widehat{\alpha_l} = \sigma_l \otimes \sigma_l$, we must construct the function $G$ such that 
\[\widehat{G(v)} = v \otimes v\]
for all $v\in TM$.
Using such function $G$, we can rewrite equation \eqref{equation: 2.intermediate_global} as
\begin{equation}
\label{equation: SDE_on mani}
\mathbf{d}X_t = \left[V(X_t) + \dfrac{1}{2} \sum_{l = 1}^p G\circ\sigma_l(X_t)\right]dt + \sum_{l=1}^p \sigma_l(X_t) dW^l_t.
\end{equation}
We have already seen an example of such function $G$ in case of standard Ito SDEs, wherein
\[G(v)|_U =  -\Gamma^i_{jk}v^jv^k\partial_i + v^iv^j\partial_{ij}.\]
This was based on construction of linear map $Q_x:\mathfrak{D}_xM\to T_xM$ from the given connection and on the isomorphism $J_x:\mathfrak{D}_xM\to T_xM \oplus (T_xM\odot T_xM)$, such that
\[G(v_x) = J_x^{-1}((v_x,v_x\otimes v_x)) = (Q_x, \widehat{\cdot})^{-1}((v_x,v_x\otimes v_x)).\]
However, this is just one of the examples. For general cases, one does not need to consider the isomorphism $J_x$ at all. One can simply consider equation \eqref{equation: SDE_on mani} and the special functions $G$ given by the following definition.
\begin{definition}
\label{def: diffusion generator}
We will call a fiber preserving map $G:TM\to \mathfrak{D}M$ over identity as a \textbf{\textit{diffusion generator}} if $\forall$ $Y\in TM$, $\widehat{G(Y)} = Y\otimes Y$. We will denote the set of all diffusion generators on $M$ by $\mathcal{G}(M)$.
\end{definition}
Therefore, if we are given an isomorphism $I_x = J^{-1}_x: T_xM\oplus (T_xM\odot T_xM)\to \mathfrak{D}_xM$, then a diffusion generator $G\in \mathcal{G}(M)$ can be identified with a map $A_x:T_xM\to T_xM$ such that
\[G(v_x) = I_x((A(v_x),v_x\otimes v_x)).\]
The map $A:TM \to TM$ need not be a bijection or a linear map.
\begin{remark}
From \cite{gliklikh2011global}, we know that the drift term of equation \eqref{equation: 2.intermediate_global}, or equivalently equation \eqref{equation: SDE_on mani}, is known as the \textit{generator} of the process $X_t\in M$. Therefore, if we consider the equation
\[\mathbf{d}Z_t = \left[\dfrac{1}{2}(G\circ\sigma)(Z_t)\right]dt + \sigma(Z_t) dW_t,\]
then the generator for $Z_t$ is $\dfrac{1}{2}G\circ\sigma$. Therefore, diffusion generator can be uniquely identified with generator of a semi-martingale driven by one dimensional Wiener process. It is for this reason that we use the word \textit{diffusion generator} to distinguish it from \textit{generator}.
\end{remark}

\subsection{Our contributions.}\label{section:contributions}
In this article we consider Stochastic Differential Equations (SDEs) from the view point of diffusion generators. As equation \eqref{equation: SDE_on mani} is just another way of looking at equation \eqref{equation: 2.intermediate_global}, in this article we are interested in using the diffusion generator approach to construct the diffusor field $\alpha$ in equation \eqref{equation: 2.intermediate_global}. To this end, we are interested in the construction of diffusion generators. In particular, we are interested in construction of diffusion generator using a regular Lagrangian. As discussed in the introductory section, we use this approach to derive generalized / extended Ito / Ito-Wentzell type formula on manifolds. Our contributions from this article are
\begin{itemize}
\item description of SDEs using regular Lagrangians, and
\item derivation of generalized / extended Ito / Ito-Wentzell type formula on manifolds from the view point of diffusion generators.
\end{itemize}

\section{Intrinsic Stochastic Differential Equations using diffusion generator.}
\label{section:main}
Before formally defining what we mean by Intrinsic SDE using diffusion generator approach, we would like to see if equation \eqref{equation: SDE_on mani} has a unique strong solution that is adapted to the filteration generated by the Wiener process $W_t\in \mathbb{R}^p$. This should not come as a surprise, because we already know that equation \eqref{equation: SDE_on mani} is just reformulation of equation \eqref{equation: 2.intermediate_global}, and that equation \eqref{equation: 2.intermediate_global} has a unique local strong solution when the coefficients are smooth. This can be stated as follows.
\begin{proposition}
Given a smooth diffusion generator $G\in \mathcal{G}(M)$, and smooth vector fields \[V, \sigma_1, \sigma_2, ..., \sigma_p\in \mathfrak{X}(M),\]
the Intrinsic SDE
\begin{equation}
\mathbf{d}X_t = \left[V(X_t) + \dfrac{1}{2} \sum_{l = 1}^p G\circ\sigma_l(X_t)\right]dt + \sum_{l=1}^p \sigma_l(X_t) dW^l_t.
\tag{\ref{equation: SDE_on mani}}
\end{equation}
has a unique local strong solution i.e. there exists a semi-martingale $X_t\in M$ that satisfies the equation locally in time, for any initial condition $X_0\in M$ in the strong sense.
\end{proposition}
\begin{proof}
Suppose for vector field $\sigma_l\in \mathfrak{X}(M)$, locally in chart $(U,\chi)$ with coordinates $(x^1, x^2, ..., x^n)$, the diffusor $\alpha_l = G(\sigma_l)$ is given as
$\tilde{\alpha_l} = G(\sigma_l)|_U = a^i_l\dfrac{\partial}{\partial x^i} + \sigma^i_l \sigma^j_l \dfrac{\partial}{\partial x^i \partial x^j}$. In chart $(U,\chi)$, the left hand side of equation \eqref{equation: SDE_on mani} is given by
\begin{equation}
\mathbf{d}X_t|_U = dX^i_t\dfrac{\partial}{\partial x^i} +\dfrac{1}{2} d[X^i_t,X^j_t]\dfrac{\partial^2}{\partial x^i \partial x^j},
\end{equation}
where $X^i_t = \chi^i(X_t)$. Therefore, in chart $(U,\chi)$, we get the Ito SDEs,
\begin{equation}
dX^i_t = (V^i+ \dfrac{1}{2}\sum_{l = 1}^p a^i_l)dt + \sigma^i_l dW^l_t
\end{equation}
and
\begin{equation}
d[X^i_t,X^j_t] = \sigma^i_l(X_t)\sigma^j_l(X_t) dt.
\end{equation}
The smoothness of the diffusors $\alpha_l$ follows from the smoothness of the map $G$ and the smoothness of the vector fields $\sigma_l$. As the Ito SDE has a unique local solution when the co-efficients are smooth, we can conclude that if equation \eqref{equation: SDE_on mani} is coordinate invariant then there exists a unique semi-martingale $X_t$ that satisfies equation \eqref{equation: SDE_on mani} locally in time. Readers interested in the proof that equation \eqref{equation: SDE_on mani} is coordinate invariant, can refer to Appendix \ref{appendix: proof of co-ordinate invariance.}.
\end{proof}
This proposition demonstrates that, using diffusion generator, it is possible to define SDEs on manifolds without explicitly mentioning the Schwartz morphism. This allows us to give the following definition.
\begin{definition}
We define \textbf{\textit{Intrinsic Stochastic Differential Equation}} using diffusion generator as a 3-tuple $(V,\{\sigma_i\}, G)$, where $V\in \mathfrak{X}(M)$, $\sigma_i\in \mathfrak{X}(M)$ for $i \in \{1, 2, ..., p\}$, and $G\in \mathcal{G}(M)$. The Intrinsic SDE $(V,\{\sigma_i\}, G)$ can also be written in form of equation \eqref{equation: SDE_on mani}. A \textbf{\textit{solution}} of the SDE $(V,\{\sigma_i\}, G)$ is a semi-martingale $X_t\in M$ that satisfies equation \eqref{equation: SDE_on mani} in all the charts in the strong sense.
\end{definition}

\subsection{Diffusion generator and Schwartz morphism.}\label{section:diffusion_generator_Schwartz morphism}
In the following lemma, we prove that it is possible to construct Schwartz morphism using $V\in \mathfrak{X}(M)$, $\sigma_i\in \mathfrak{X}(M)$ for $i \in \{1, 2, ..., p\}$, and diffusor generator $G\in \mathcal{G}(M)$. However, as we shall see latter, the converse is not true.
\begin{lemma}
\label{lemma: S eq S}
For every vector fields $V\in \mathfrak{X}(M)$, $\sigma_i\in \mathfrak{X}(M)$ for $i \in \{1, 2, ..., p\}$, and diffusor generator $G\in \mathcal{G}(M)$, there exists a Schwartz morphism $\beta(y,x):\mathfrak{D}_y\mathbb{R}^{p+1} \to \mathfrak{D}_xM$ such that \[\mathbf{d}X_t = \beta((t,W_t),X_t) \mathbf{d}(t,W_t) = \left[V + \dfrac{1}{2} \sum_{l = 1}^p G(\sigma_l)\right]dt + \sum_{l=1}^p \sigma_l dW^l_t.\]
\end{lemma}
\begin{proof}
Given vector fields $V\in \mathfrak{X}(M)$, $\sigma_i\in \mathfrak{X}(M)$ for $i \in \{1, 2, ..., p\}$, and diffusor generator $G\in \mathcal{G}(M)$ that is locally given as
\[G(\sigma)|_U = g^i(\sigma)\partial_i + \sigma^i\sigma^j\partial^2_{ij};\]
we can consider the Schwartz morphism $\beta(y,x):\mathfrak{D}_y\mathbb{R}^{p+1} \to \mathfrak{D}_xM$ such that locally it is given as
\[\beta(y,x) L|_U = \left( V^i(x) a^0 + \sigma^i_l(x) a^l + \sum_{n = 1}^p \dfrac{1}{p} g^i(\sigma_n(x))\delta_{lm}b^{lm}\right) \partial_i + \left(\sigma^i_l(x) \sigma^j_m(x) b^{lm} \right) \partial^2_{ij},\]
for every $L \in \mathfrak{D}_y\mathbb{R}^{p+1}$ such that $L = a^k\partial_k + b^{kz}\partial^2_{kz}$ and the indices $k,z \in \{0,1, 2, ..., p\}$ and $l,m \in \{1, 2, ..., p\}$.
Clearly, this Schwartz morphism is constructed using the local component of the vector fields and the diffusion generator. It can be verified that $\beta((t,W_t),X_t) \mathbf{d}(t,W_t)$ is locally given as
\[\beta((t,W_t),X_t) \mathbf{d}(t,W_t)|_U = \left[V^i\partial_i + \dfrac{1}{2} \sum_{l = 1}^p g^i(\sigma_l)\partial_i + \dfrac{1}{2} \sum_{l = 1}^p\sigma^i_l\sigma^j_l\partial^2_{ij}\right]dt + \sum_{l=1}^p \sigma^i_l\partial_i dW^l_t,\]
\[ = \left[V^i\partial_i + \dfrac{1}{2} \sum_{l = 1}^p G(\sigma_l)|_U\right]dt + \sum_{l=1}^p \sigma^i_l\partial_i dW^l_t.\]
\end{proof}
However, the converse of lemma \ref{lemma: S eq S} is not true. Suppose $X_t$ is a semi-martingale that satisfies
\[\mathbf{d}X_t = \beta((t,W_t),X_t) \mathbf{d}(t,W_t)\]
for some arbitrary Schwartz morphism $\beta(y,x):\mathfrak{D}_y\mathbb{R}^{p+1} \to \mathfrak{D}_xM$. Following the discussion in section \ref{section:pre}, from equation \eqref{equation: 2.intermediate} we know that locally,
\begin{equation}
\mathbf{d}X_t|_U =\left[ f^i_0(X_t)\partial_i +\dfrac{1}{2}\left( \sum_{l=1}^p g^i_{ll}(X_t)\partial_i +(f^i_l(X_t) f^j_l(X_t))\partial^2_{ij}\right)\right] dt + \sum_{l=1}^p(f^i_l(X_t) \partial_i)dW^l_t,
\end{equation}
where $f^i_l, g^i_{lm}$ are local coefficients of $\beta$. If we consider diffusion generators $G_n\in \mathcal{G}(M)$ such that they are locally given as
\[G_n(v)|_U = h^i_{n}(v)\partial_i +(v^i v^j)\partial^2_{ij},\]
for all $v\in TM$, where $h^i_n= g^i_{nn}\circ \tau_M$;
then
\[\mathbf{d}X_t|_U = \left[ f^i_0\partial_i + \dfrac{1}{2}\left( \sum_{l=1}^p G_l(f^i_l\partial_i)|_U\right)\right] dt + \sum_{l=1}^p f^i_l\partial_i dW^l_t.\]
Hence, we need multiple diffusion generators to retrieve the Schwartz morphism. However, we do not consider SDEs with multiple diffusion generators. 
\subsection{Construction of diffusion generators using flow of differential equations.}
\label{section:construction_using_flow}
As a consequence of lemma \ref{lemma: S eq S}, the construction of a Schwartz morphism boils down to the construction of a diffusion generator. In section \ref{section:pre}, we have seen that a diffusion generator $G\in \mathcal{G}(M)$ can be identified with a map $A_x:T_xM\to T_xM$ for the given isomorphism $I_x:T_xM\oplus (T_xM\odot T_xM)\to \mathfrak{D}_xM$ such that
\[G(\sigma_x) = I_x(A_x(\sigma_x), \sigma_x\otimes\sigma_x)\]
for all $\sigma_x\in T_xM$. Therefore, to construct a diffusion generator, one needs the isomorphism $I_x:\mathfrak{D}_xM\to T_xM \oplus (T_xM\odot T_xM)$ and the map $A:TM\to TM$ to define a diffusion generator. The map $A:TM\to TM$ can even be an identity map or some linear map that is easy to construct. This implies that the construction of diffusion generator depends on the given isomorphism $I_x$. Unfortunately, other than the case of manifolds with connection, which gives us the standard Ito SDEs, we do not know of any other way of obtaining this isomorphism. Hence, we do not take this route to construct the diffusion generator.

In this section we will demonstrate that it is possible to obtain diffusion generator using the flow of first order or second order differential equation. For this purpose we consider smooth curve $c(t)$ with the diffusor
\[\dfrac{\mathbf{d}c}{dt} = \mathfrak{Dc}\dfrac{d^2}{dt^2}.\]
We know that in chart $(U,\chi)$,
\begin{equation}
\dfrac{\mathbf{d}c}{dt}\Big\vert_{U} = \ddot{c}^i\partial_i + \dot{c}^i\dot{c}^j\partial^2_{ij}.
\end{equation}
Moreover, we find that $\widehat{\dfrac{\mathbf{d}c}{dt}} = \dot{c}\otimes \dot{c}$. This means that any function that maps the vector $\dot{c}$ to the diffusor $\mathbf{d}c/dt$ should give us the diffusion generator. This approach of constructing diffusion generator using smooth curves is similar to the 2-jet approach discussed in \cite{armstrong2018intrinsic}. This is because both the approaches are fundamentally based on the idea of considering upto second derivative of the curve. In this section we will only consider curves obtained by flow of first order and second order differential equations.

\subsubsection{Construction of diffusion generator using flow of first order differential equations and its relation to Stratonovich SDE.}
\begin{lemma}
\label{lemma:coordinate_invariance_of_diffusor_field}
For every vector field $\sigma\in \mathfrak{X}(M)$ there exists a unique diffusor field $\alpha\in \Gamma(\mathfrak{D}M)$ such that locally, in chart $(U,\chi)$ with coordinates $(x^1, x^2, ..., x^n)$, \begin{equation}\tilde{\alpha} = \alpha|_U = d\sigma^i\cdot\sigma\dfrac{\partial}{\partial x^i} + \sigma^i \sigma^j\dfrac{\partial^2}{\partial x^i \partial x^j},\end{equation}
where $\sigma^i = d\chi^i\cdot \sigma$.
\end{lemma}
\begin{proof}
To prove that $\alpha\in \Gamma(\mathfrak{D}M)$ is a diffusor field, we need to prove that it is coordinate invariant.
This can be achieved by considering another chart $(U,\Upsilon)$ with coordinates $(z^1,z^2, ..., z^n)$.
In chart $(U,\Upsilon)$,
\begin{equation}\alpha ' = \alpha|_U = d\breve{\sigma}^i\cdot{\sigma}\dfrac{\partial}{\partial z^i} + \breve{\sigma}^i \breve{\sigma}^j\dfrac{\partial^2}{\partial z^i \partial z^j,}\end{equation}
where $\breve{\sigma}^i = d\Upsilon^i\cdot{\sigma}$. The smoothness of the diffusor field in the chart follows from the smoothness of the vector fields.
\begin{subequations}
\begin{equation}\implies \alpha '= d(d\Upsilon^i\cdot {\sigma})\cdot{\sigma}\dfrac{\partial}{\partial z^i} + (d\Upsilon^i\cdot {\sigma}) (d\Upsilon^j\cdot {\sigma})\dfrac{\partial^2}{\partial z^i \partial z^j}\end{equation}
\begin{equation}= \left(\dfrac{\partial}{\partial x^l}\left(\dfrac{\partial\Upsilon^i}{\partial x^j}\sigma^j\right)\sigma^l\right)\dfrac{\partial}{\partial z^i} + \left(\dfrac{\partial\Upsilon^i}{\partial x^l}\sigma^l \dfrac{\partial\Upsilon^j}{\partial x^k}\sigma^k\right)\dfrac{\partial^2}{\partial z^i \partial z^j}\end{equation}
\begin{equation}= \left(\dfrac{\partial^2\Upsilon^i}{\partial x^j\partial x^l}\sigma^j\sigma^l + \dfrac{\partial\Upsilon^i}{\partial x^j}\dfrac{\partial \sigma^j}{\partial x^l}\sigma^l\right)\dfrac{\partial}{\partial z^i} + \left(\dfrac{\partial\Upsilon^i}{\partial x^l}\sigma^l \dfrac{\partial\Upsilon^j}{\partial x^k}\sigma^k\right)\dfrac{\partial^2}{\partial z^i \partial z^j}\end{equation}
\begin{equation} = \mathfrak{D}\Upsilon (\tilde{\alpha}).\end{equation}
\end{subequations}
Therefore, there exists a diffusor field $\alpha\in \Gamma(\mathfrak{D}M)$ such that locally, in chart $(U,\chi)$ with coordinates $(x^1, x^2, ..., x^n)$, \begin{equation}\alpha|_U = d\sigma^i\cdot\sigma\dfrac{\partial}{\partial x^i} + \sigma^i \sigma^j\dfrac{\partial^2}{\partial x^i \partial x^j}.\end{equation}
\end{proof}

\begin{lemma}
\label{lemma:existence_of_diffusion_generator}
There exists a unique diffusion generator $G_S\in\mathcal{G}(M)$ on the manifold $M$ such that the solution of the ODE $\dot{x} = \sigma(x)$ is also the solution of the Schwartz differential equation
\begin{equation}\dfrac{\mathbf{d}x}{dt} = G_S(\sigma(x)),\end{equation}
where $\sigma\in \mathfrak{X}(M)$.
\end{lemma}
\begin{proof}
If there exists a diffusion generator $G_S\in \mathcal{G}(M)$, then in chart $(U,\chi)$,
\begin{equation}G_S(\sigma)|_U = a^i\dfrac{\partial}{\partial x^i} + \sigma^i\sigma^j\dfrac{\partial^2}{\partial x^i \partial x^j}.\end{equation}
If $x(t)$ is the solution for the ODE $\dot{x} = \sigma(x)$, then locally
\begin{equation}\mathbf{d}x/dt|_U = \dfrac{d^2}{dt^2}(\chi^i\circ x)\dfrac{\partial}{\partial x^i} + \sigma^i\sigma^j\dfrac{\partial^2}{\partial x^i \partial x^j}.\end{equation}
If the equation $\dfrac{\mathbf{d}x}{dt} = G_S(\sigma)$ is satisfied by $x(t),$ then
\begin{equation}a^i = \dfrac{d^2}{dt^2}(\chi^i\circ x)=\dfrac{d}{dt}(d\chi^i\cdot \sigma) = \dfrac{d\sigma^i}{dt} = d\sigma^i\cdot \sigma.\end{equation}
Therefore,\begin{equation}G_S(\sigma)|_U = d\sigma^i\cdot \sigma\dfrac{\partial}{\partial x^i} + \sigma^i\sigma^j\dfrac{\partial^2}{\partial x^i \partial x^j}.\end{equation}
From lemma \ref{lemma:coordinate_invariance_of_diffusor_field} we know that $G_S(\sigma)$ is a diffusor and the above equation is its local representation. Conversely, if we consider the diffusor $G_S(\sigma)$ such that $G_S(\sigma)|_U = d\sigma^i\cdot \sigma\dfrac{\partial}{\partial x^i} + \sigma^i\sigma^j\dfrac{\partial^2}{\partial x^i \partial x^j}$, then the solution of the ODE $\dot{x} = \sigma(x)$ is the same as the solution of the Schwartz ODE
$\dfrac{\mathbf{d}x}{dt} = G_S(\sigma(x))$. The uniqueness follows due to the fact that $\widehat{G_S(\sigma)} = \sigma\otimes\sigma$.
\end{proof}

Lemma \ref{lemma:existence_of_diffusion_generator} can be interpreted as a special case of a more general result presented in theorem \ref{theorem:stratonovich_to_Schwartz} (or theorem 7.22 of Emery's book \cite{emery2012stochastic}). Infact, the SDE obtained using the diffusion generator $G_S\in \mathcal{G}(M)$ is an Intrinsic representation of a Stratonovich SDE. This is because if we consider the SDE
\begin{equation}
\mathbf{d}X_t = \left[V + \dfrac{1}{2} \sum_{l = 1}^p G_S(\sigma_l)\right]dt + \sum_{l=1}^p \sigma_l dW^l_t,
\end{equation}
then we see that the local Ito SDE
\begin{equation}dX_t^i = \left[ V^i+\dfrac{1}{2}\sum_{l=1}^{p}\dfrac{\partial \sigma^i_l}{\partial x^j} \sigma^j_l\right] dt + \sum_{l=1}^{p} \sigma_l^i dW^l_t,\end{equation}
is the same as the Stratonovich SDE,
\begin{equation}\delta X_t^i = V^i dt + \sum_{l=1}^{p} \sigma_l^i\circ  dW^l_t.\end{equation}
Therefore, we use the subscript $S$ to indicate the special diffusion generator $G_S$, which can convert the Stratonovich SDE $(V,\{\sigma_1, ..., \sigma_p\})$ into Intrinsic SDE $(V,\{\sigma_1, ..., \sigma_p\}, G_S)$. 
\begin{definition}
\label{def: stratonovich diffusion generator}
The unique diffusion generator $G_S\in\mathcal{G}(M)$ that ensures that the solution of the ODE $\dot{x} = \sigma(x)$ is also the solution of the Schwartz differential equation
\begin{equation}\dfrac{\mathbf{d}x}{dt} = G_S(\sigma(x)),\end{equation}
where $\sigma\in \mathfrak{X}(M)$, will be called \textbf{\textit{Stratonovich diffusion generator}}.
\end{definition}

\subsubsection{Construction of diffusion generator using flow of second order differential equations and its relation to Ito SDE.}
We have already seen that the diffusion generator obtained using the first order vector field results in the Intrinsic representation of the Stratonovich SDE. Now, we will try to construct the diffusion generator using second order differential equations. A second order differential equations is defined by a vector field $Z$ on the tangent bundle $TM$ such that $T\tau_M\circ Z = Id_{TM}$.

\begin{lemma}
\label{lemma: second_order_construction}
For a given second order differential equation $Z\in \mathfrak{X}(TM)$, there exists a diffusion generator $G_{Z}\in \mathcal{G}(M)$ such that if $z(t)$ is the solution of the second order differential equation $\dot{z} = Z(z)$, then \begin{equation}\dfrac{\mathbf{d}x}{dt} = G_{Z}(z(t)),\end{equation}
where $x(t) = \tau_M(z(t))$.
\end{lemma}
\begin{proof}
Every second order vector field is locally given as $Z((x,v)) = ((x,v),(v,Z_V(x,v)))$ for all $z = (x,v)\in TM$, where $Z_V(z)\in VTM$ with $VTM$ as the vertical bundle. As $x(t) = \tau_M(z(t))$, \[\dot{x}(t) = T\tau_M(z(t))\cdot \dot{z}(t) =T\tau_M(z(t))\cdot Z(z(t)) = z(t).\]
Therefore, $\ddot{x}^i(t) = Z^i_V(z(t))$. Since
\[\dfrac{\mathbf{d}x}{dt}\Big\vert_{U} = \ddot{x}^i\partial_i + \dot{x}^i\dot{x}^j\partial^2_{ij},\]
\[\dfrac{\mathbf{d}x}{dt}\Big\vert_{U} = Z_V^i(z(t))\partial_i + z^i(t)z^j(t)\partial^2_{ij}.\]
Therefore if $x(t) = \tau_M(z(t))$, $\dot{z} = Z(z)$, and 
\[G_Z(\sigma)|_U = Z_V^i(\sigma)\partial_i + \sigma^i \sigma^j \partial^2_{ij};\]
then
\[\dfrac{\mathbf{d}x}{dt} = G_{Z}(z(t)).\]
\end{proof}

In terms of the covariant derivative $\nabla$, the second order equations are given as $\nabla_{\dot{x}}\dot{x} = Y(x)$, for some $Y\in \mathfrak{X}(M)$. We know that for $\nabla_{\dot{x}}\dot{x} = Y(x)$, \[\ddot{x}^i(t) = Y^i(x) -\Gamma^i_{jk}\dot{x}^j\dot{x}^k.\] The diffusion generator can now be constructed directly using lemma \ref{lemma: second_order_construction}. Given a connection on the manifold, in local coordinates $(U,\chi)$ the diffusion generator is given as,
\begin{equation}
G(\dot{x})|_U = \ddot{x}^i\partial_i + \dot{x}^i\dot{x}^j\partial^2_{ij}= Y^i(x)-\Gamma^i_{jk}\dot{x}^j\dot{x}^k\partial_i + \dot{x}^i\dot{x}^j\partial^2_{ij}.
\end{equation}
$Y = 0$ is a special case, in which the solution curve is a geodesic. If we consider $Y = 0$, then we find that the resulting SDE is the Intrinsic representation of the Ito SDE on manifold with a connection, as defined in \cite{gliklikh2011global} and \cite{emery2012stochastic}.
\begin{definition}
\label{def: Ito diffusion generator}
Let $G_I\in\mathcal{G}(M)$ on the manifold $M$ be a diffusion generator such that the solution of the differential equation $\nabla_{\dot{x}}\dot{x} = 0$ is also the solution of the Schwartz equation $\mathbf{d}x/dt = G_I(\dot{x})$. Then $G_I\in\mathcal{G}(M)$ will be called \textbf{\textit{Ito diffusion generator}}.
Locally, in chart $(U,\chi)$,
\begin{equation}
\label{equation: Ito diffusion generator}
G_I(v)|_U = -\Gamma_{ij}^k v^iv^j\dfrac{\partial}{\partial x^i} + v^iv^j\dfrac{\partial^2}{\partial x^i \partial x^j},
\end{equation}
for all $v\in TM$. We will call an SDE generated by $G_I$ as \textbf{\textit{standard Ito SDEs}}.
\end{definition}
In this context, since $\mathbf{d}x/dt = G_I(\dot{x})$, the Ito diffusion generator $G_I$ is just another way of looking at the geodesic spray. Therefore, in order to use the Ito diffusion generator to obtain the standard Ito SDE, the manifold must be equipped with a connection. In the following section we show that if regular Lagrangian is used to define second order equation, then using lemma \ref{lemma: second_order_construction} it is possible to construct diffusion generator without using the connection.

\section{Construction of diffusion generator using Lagrangian.}
\label{section:construction using Lagrangians}
In order to describe Hamiltonian equations on the tangent bundle $T^*M$, it is enough to specify a function $H\in \mathfrak{F}(T^*M)$. Dually, for Hamiltonian equations on $TM$, it we need a function $L\in \mathfrak{F}(TM)$. This function is known as a Lagrangian. From elementary mechanics, we know that if the fiber derivative of the Lagrangian $FL:TM\to T^*M$ is a regular function, then the Lagrangian vector field on $TM$ gives a second order equation. Such Lagrangians are called \textit{regular Lagrangians}. Therefore, even if the manifold does not have a connection, a manifold with a regular Lagrangian should be enough to construct a diffusion generator on the manifold. Readers may refer to chapter 3 of \cite{abraham2008foundations} for basic definitions/results in mechanics. The following proposition states the existence of a diffusion generator for every regular Lagrangian.
\begin{proposition}
\label{proposition:diffusion generator from Lagrangian}
For every regular Lagrangian $L\in \mathfrak{F}(TM)$ there exists a diffusion generator $G_L\in \mathcal{G}(M)$ such that if $z(t)$ is the solution of the Hamiltonian dynamics $\dot{z} = \omega_L^\sharp dE$  (where $\omega_L = FL^*\omega_0$, $\omega_0$ is the canonical symplectic form on $T^*M$, and $E\in \mathfrak{F}(TM)$ such that $E(v) = FL(v)\cdot v - L(v)$), then
\begin{equation}\mathbf{d}x/dt = G_L(z(t)),
\end{equation}
where $x(t) = \tau_M(z(t))$. Moreover, locally in chart $(U,\chi)$,
\begin{equation}
\label{equation:diffusion_generator for arbitrary Lagrangian.}
G_L(\sigma)|_U = \left[\left\lbrace\dfrac{\partial^2 L}{\partial \dot{x}^i \partial \dot{x}^j}\right\rbrace^{-1}\left(\dfrac{\partial L}{\partial x^j} - \dfrac{\partial^2 L}{\partial x^k\partial \dot{x}^j}\sigma^k \right)\right]\dfrac{\partial}{\partial x^i} + \sigma^i \sigma^j\dfrac{\partial^2}{\partial x^i \partial x^j},
\end{equation}
for all $\sigma\in T_xM$.
\end{proposition}
\begin{proof}
From basic mechanics we know that in local coordinates, the equation $\dot{z} = \omega_L^\sharp dE$ with initial condition $z(0) = (x_0,v_0)$ is equivalent to the Euler-Lagrange equation $\dfrac{d}{dt}\dfrac{\partial L}{\partial \dot{x}^i} = \dfrac{\partial L}{\partial x^i}$ with initial condition $x(0) = x_0$ and $\dot{x}(0) = v_0$. Since the Lagrangian is regular, inverse of $\dfrac{\partial^2 L}{\partial \dot{x}^i \partial \dot{x}^j}$ exists (proposition 3.5.10 in \cite{abraham2008foundations}).
\begin{equation}
\therefore \ddot{x}^i(t) = \left\lbrace\dfrac{\partial^2 L}{\partial \dot{x}^i \partial \dot{x}^j}\Big\vert_{z(t)}\right\rbrace^{-1}\left(\dfrac{\partial L}{\partial x^j}\Big\vert_{z(t)} - \dfrac{\partial^2 L}{\partial x^k\partial \dot{x}^j}\Big\vert_{z(t)}\dot{x}^k(t) \right).
\end{equation}
From lemma \ref{lemma: second_order_construction}, we know that if $G_L\in \mathcal{G}(M)$, such that locally in chart $(U,\chi)$,
\[G_L(\sigma)|_U = \left[\left\lbrace\dfrac{\partial^2 L}{\partial \dot{x}^i \partial \dot{x}^j}\Big\vert_{(x,\sigma)}\right\rbrace^{-1}\left(\dfrac{\partial L}{\partial x^j}\Big\vert_{(x,\sigma)} - \dfrac{\partial^2 L}{\partial x^k\partial \dot{x}^j}\Big\vert_{(x,\sigma)}\sigma^k \right)\right]\dfrac{\partial}{\partial x^i}\]
\begin{equation}
+ \sigma^i \sigma^j\dfrac{\partial^2}{\partial x^i \partial x^j},
\end{equation}
for all $\sigma\in T_xM$, then
\[\mathbf{d}x/dt = G_L(z(t)),\]
where  $x(t) = \tau_M(z(t))$ and $z(t)$ is the solution of $\dot{z} = \omega_L^\sharp dE$.
\end{proof}

\begin{definition}
\label{def: lagrangian diffusion generator}
Let $G_L\in\mathcal{G}(M)$ be a diffusion generator such that the solution $z(t)$ of the Hamiltonian dynamics $\dot{z} = \omega_L^\sharp dE$ (where $\omega_L = FL^*\omega_0$, $\omega_0$ is the canonical symplectic form on $T^*M$, and $E\in \mathfrak{F}(TM)$ such that $E(v) = FL(v)\cdot v - L(v)$), ensures that $\mathbf{d}x/dt = G_L(z(t))$, where $x(t) = \tau_M(z(t))$. Then $G_L\in\mathcal{G}(M)$ will be called \textbf{\textit{Lagrangian diffusion generator}}. We will say that an SDE is generated by a Lagrangian $L$, if $G_L$ is the diffusion generator.
\end{definition}

\subsection{Special cases.}
\begin{enumerate}[label=\Roman*.]
\item \textbf{Manifold $\mathbf{M}$ with a non-degenerate $\mathbf{T^0_2M}$ tensor-field $\mathbf{\alpha}$.}
As $\alpha\in T^0_2M$ is non-degenerate, if $L\in \mathcal(TM)$ such that \begin{equation}L(v) = \dfrac{1}{2}\alpha(v,v),\end{equation}
for all $v\in TM$, then from proposition \ref{proposition:diffusion generator from Lagrangian} we know that
\[G_L(\sigma_x)|_U = \left[\left\lbrace\dfrac{\partial^2 L}{\partial \dot{x}^i \partial \dot{x}^j}\Big\vert_{(x,\sigma)}\right\rbrace^{-1}\left(\dfrac{\partial L}{\partial x^j}\Big\vert_{(x,\sigma)} - \dfrac{\partial^2 L}{\partial x^k\partial \dot{x}^j}\Big\vert_{(x,\sigma)}\sigma^k \right)\right]\dfrac{\partial}{\partial x^i}\]
\begin{equation}+ \sigma^i \sigma^j\dfrac{\partial^2}{\partial x^i \partial x^j}.\end{equation}
Therefore,
\begin{equation}
\label{equation: diffusion generator with quadratic lagrangian}
G_L(\sigma_x)|_U = \left[\alpha^{ij}\left(\dfrac{1}{2}\dfrac{\partial \alpha_{lm}}{\partial x^j} \sigma^l\sigma^m - \dfrac{\partial \alpha_{jm}}{\partial x^k} \sigma^k\sigma^m \right)\right]\dfrac{\partial}{\partial x^i} + \sigma^i \sigma^j\dfrac{\partial^2}{\partial x^i \partial x^j}.
\end{equation}
\item \textbf{Riemannian manifold, $\mathbf{(M,g)}$, with Kinetic energy as the Lagrangian.}
A special case of proposition \ref{proposition:diffusion generator from Lagrangian}, is a regular Lagrangian $L\in \mathfrak{F}(TM)$ such that \begin{equation}L(v) = \dfrac{1}{2} g^\flat v\cdot  v,\end{equation}
where $g$ is the Riemannian metric on the manifold $M$.
In Mechanics, such a Lagrangian is called the Kinetic Energy.
From basic mechanics we know that if the initial condition is $v\in TM$ and the solution is given by $z(t)$, then $x(t) = \tau_M(z(t))$ is a geodesic in the direction of $v\in TM$ i.e. $x(t) = exp_{\tau_M(v)}(vt) = exp_{x_0}(\sigma(x_0)t).$

From Riemannian geometry it is known that
\begin{equation}
\dfrac{d}{dt}\Big\vert_{t=0}(exp_{\tau_M(v)}(vt)) = v
\end{equation}
and, locally in chart $(U,\chi)$,
\begin{equation}
\dfrac{d^2}{dt^2}\Big\vert_{t=0}(exp^k_{\tau_M(v)}(vt)) = \left\langle v,\nabla_v g^\sharp d\chi^k \right\rangle = -\Gamma_{ij}^k v^iv^j;
\end{equation}
where $exp^k = \chi^k\circ exp$. From the proof of lemma \ref{lemma:coordinate_invariance_of_diffusor_field}, we can conclude that there exists a function $G\in \mathcal{G}(M)$ such that locally
\begin{equation}
\label{equation: KE as Lagrangian}
G(v)|_U = -\Gamma_{ij}^k v^iv^j\dfrac{\partial}{\partial x^i} + v^iv^j\dfrac{\partial^2}{\partial x^i \partial x^j}.
\end{equation}
Comparing equation \eqref{equation: KE as Lagrangian} with equation \eqref{equation: Ito diffusion generator}, we notice that this is a special case of diffusion generator constructed using connection obtained from the Riemannian metric. Hence, this is the Ito diffusion generator on the Riemannian manifold.

\item \textbf{Riemannian manifold, $\mathbf{(M,g)}$, with Kinetic energy - Potential Energy as the Lagrangian.}
Let $\Phi:M\to R$ be the potential energy. Therefore, the Lagrangian is given by $L\in\mathfrak{F}(TM)$ such that \begin{equation}L(v) = \dfrac{1}{2}g^\flat v\cdot v  - \Phi(\tau_M(v)).\end{equation}
Using proposition \ref{proposition:diffusion generator from Lagrangian} we get
\[G_L(\sigma_x)|_U = \left[\left\lbrace\dfrac{\partial^2 L}{\partial \dot{x}^i \partial \dot{x}^j}\Big\vert_{(x,\sigma)}\right\rbrace^{-1}\left(\dfrac{\partial L}{\partial x^j}\Big\vert_{(x,\sigma)} - \dfrac{\partial^2 L}{\partial x^k\partial \dot{x}^j}\Big\vert_{(x,\sigma)}\sigma^k \right)\right]\dfrac{\partial}{\partial x^i}\]
\begin{equation}+ \sigma^i \sigma^j\dfrac{\partial^2}{\partial x^i \partial x^j}.\end{equation}
Therefore,
\[G_L(\sigma_x)|_U = g^{ij}(x)\left(\dfrac{\sigma^l}{2}\dfrac{\partial g_{lm}}{\partial x^j}(x)\sigma^m-\dfrac{\partial \Phi}{\partial x^j}(x) - \dfrac{\partial g_{jm}}{\partial x^k} \sigma^k\sigma^m\right)\dfrac{\partial}{\partial x^i}\]
\begin{equation}+ \sigma^i \sigma^j\dfrac{\partial^2}{\partial x^i \partial x^j}.\end{equation}
In other words,
\begin{equation}G_L(\sigma_x)|_U =\left( -\Gamma^i_{jk}\sigma^j\sigma^k -g^{ij}(x)\dfrac{\partial \Phi}{\partial x^j}(x)\right) \dfrac{\partial}{\partial x^i} + \sigma^i \sigma^j\dfrac{\partial^2}{\partial x^i \partial x^j}.\end{equation}
\end{enumerate}

\subsection{Example.}
Let us consider the state space to be $\mathbb{R}^2$ with coordinates $(x,y)$. A Lagrangian $L$ is defined as \[L(x,y,v_1,v_2) = v_1^4 + v_1^2 + v_1 + v_2 + v_2^2 + v_2^4 - x^2 - y^2.\] From equation \eqref{equation:diffusion_generator for arbitrary Lagrangian.} we know that as the Lagrangian is regular
\begin{equation}G_L = \left(\begin{bmatrix}
    \dfrac{-x}{6v_1^2 + 1}\\ \dfrac{-y}{6v_2^2 + 1}
\end{bmatrix} , \begin{bmatrix}
    v_1v_1 & v_1v_2\\ v_1v_2 & v_2v_2
\end{bmatrix}\right).\end{equation}
We will take drift to be \begin{equation}V(x,y) = \begin{bmatrix}
    1\\ sin(5\pi x)
\end{bmatrix}\end{equation}
and the noise vectors as
\begin{equation}\sigma_1(x,y) = \begin{bmatrix}
    y\\ 0
\end{bmatrix},\end{equation}
and 
\begin{equation}\sigma_2(x,y) = \begin{bmatrix}
    0\\ y
\end{bmatrix}.\end{equation}
Therefore, the Intrinsic SDE is given by
\begin{equation}\mathbf{d} \begin{bmatrix}
    x\\ y
\end{bmatrix} = \left[\begin{bmatrix}
    1\\ sin(5\pi x)
\end{bmatrix} + \dfrac{1}{2} G_L(\sigma_1) + \dfrac{1}{2} G_L(\sigma_2)\right] dt + \begin{bmatrix}
    y\\ 0
\end{bmatrix} dW^1_t + \begin{bmatrix}
    0\\ y
\end{bmatrix} dW^2_t.\end{equation}
As $\mathbf{d}X_t = \left(dX_t,\dfrac{1}{2}d[X_t,X_t]\right)$, we can say that the underlying Ito SDE for the current example is given as
\[d\begin{bmatrix}
    x\\ y
\end{bmatrix} = \left[\begin{bmatrix}
    1\\ sin(5\pi x)
\end{bmatrix} + \dfrac{1}{2} \begin{bmatrix}
    \dfrac{-x}{6y^2 + 1}\\ -y
\end{bmatrix} + \dfrac{1}{2} \begin{bmatrix}
    -x\\ \dfrac{-y}{6y^2 + 1}
\end{bmatrix}\right] dt\]
\begin{equation}
+ \begin{bmatrix}
    y\\ 0
\end{bmatrix} dW^1_t + \begin{bmatrix}
    0\\ y
\end{bmatrix} dW^2_t.
\end{equation}
On the other hand, the standard Ito SDE representation will depend on the metric on $\mathbb{R}^2$ to define a connection on $\mathbb{R}^2$.

\section{Some equivalent representations and extended Ito formulae.}
\label{section:last}
In this section we show that we can convert the Intrinsic SDE into an equivalent Belopolskya-Daletskii type SDE. In order to get the Belopolskya-Daletskii form for the given Intrinsic SDE, we first convert the given Intrinsic SDE into standard Ito SDE and then consider the Belopolskya-Daletskii form for the resulting standard Ito SDE. The idea of converting Intrinsic SDEs into standard Ito SDE/Stratonovich SDE and vice-versa is well-known. What we consider here are equivalent representations of Intrinsic SDEs obtained using the diffusion generator approach. Furthermore, based on these conversion formulae, we give the extended Ito type formula for Intrinsic SDEs on manifolds.

\subsection{Equivalent representations of Intrinsic SDEs in standard Ito representation, Stratonovich representation, and Belopolskya-Daletskii form.}
\label{section:conversion}
Earlier, in section \ref{section:construction using Lagrangians} we have observed that the standard Ito SDE \[\left(V,\{\sigma_1, ..., \sigma_p\}\right),\] is the same as the Intrinsic SDE \[\left(V,\{\sigma_1, ..., \sigma_p\}, G_I\right).\] However, we do not know if Intrinsic SDEs with arbitrary diffusion generator G can have a standard Ito representation. It seems reasonable that the Intrinsic SDE \[\left(V,\{\sigma_1, ..., \sigma_p\}, G\right)\] is the same as the standard Ito SDE
\[\left(V + \dfrac{1}{2}\sum_{l = 1}^p (G(\sigma_l) - G_I(\sigma_l)),\{\sigma_1, ..., \sigma_p\}\right).\] However, we need to prove that $G(\sigma_l) - G_I(\sigma_l)$ is indeed a tangent vector.
\begin{lemma}
\label{lemma: difference of diffusors}
For every two diffusion generators $G, G_\alpha \in \mathcal{G}(M)$, there exists a fiber preserving map $\nabla_\alpha^G:TM\to TM$ over identity such that $\nabla_\alpha^G(X) = G(X) - G_\alpha(X)$ $\forall$ $X\in TM$.
\end{lemma}
\begin{proof}
As per the definition of diffusion generator, for any  $G\in \mathcal{G}(M)$, $\widehat{G(X)} = X\otimes X$, $\forall$ $X\in TM$. Therefore, $\widehat{G(X)-G_\alpha(X)} = 0$ i.e., $G(X)-G_\alpha(X)\in TM$ $\forall$ $X\in TM$.
\end{proof}

\begin{lemma}
$\left(V,\{\sigma_1, ..., \sigma_p\}, G\right)$ is equivalent to \[\left(V + \dfrac{1}{2}\sum_{l = 1}^p \nabla^G_\alpha(\sigma_l),\{\sigma_1, ..., \sigma_p\}, G_\alpha \right).\]
\end{lemma}
\begin{proof}
\begin{equation}\mathbf{d}X_t = V dt + \dfrac{1}{2} \sum_{l = 1}^p G(\sigma_l)dt + \sum_{l=1}^p \sigma_l(x) dW^l_t\end{equation}
\begin{equation} = V dt + \dfrac{1}{2} \sum_{l = 1}^p \left(\nabla^G_\alpha(\sigma_l) + G_\alpha(\sigma_l)\right) dt + \sum_{l=1}^p \sigma_l(x) dW^l_t\end{equation}
From lemma \ref{lemma: difference of diffusors} we know that $\nabla^G_\alpha(\sigma_l)$ is a vector. Hence,
\begin{equation}\mathbf{d}X_t = \left[V + \dfrac{1}{2} \sum_{l = 1}^p \nabla^G_\alpha(\sigma_l)\right]dt + \dfrac{1}{2} \left(\sum_{l = 1}^pG_\alpha(\sigma_l)\right) dt + \sum_{l=1}^p \sigma_l(x) dW^l_t\end{equation}
can be considered as the SDE $\left(V + \dfrac{1}{2}\sum_{l = 1}^p \nabla^G_\alpha(\sigma_l),\{\sigma_1, ..., \sigma_p\}, G_\alpha \right)$.
\end{proof}

\noindent Due to this lemma, if the manifold is equipped with a connection, then the Intrinsic SDE $\left(V,\{\sigma_1, ..., \sigma_p\}, G\right)$ has the standard Ito representation
\begin{equation}\left(V + \dfrac{1}{2}\sum_{l = 1}^p \nabla^G_I(\sigma_l),\{\sigma_1, ..., \sigma_p\}\right).\end{equation}
Similarly, the Intrinsic SDE $\left(V,\{\sigma_1, ..., \sigma_p\}, G\right)$ has the Stratonovich representation
\begin{equation}\left(V + \dfrac{1}{2}\sum_{l = 1}^p \nabla^G_S(\sigma_l),\{\sigma_1, ..., \sigma_p\}\right).\end{equation}
From \cite{gliklikh2011global}, we know that the Belopolskya-Daletskii form for the standard Ito SDE $\left(V,\{\sigma_1, ..., \sigma_p\}, G_I\right)$ is given by
\begin{equation}dX_t = exp_{X_t}\left(V(X_t)dt + \sum_{l = 1}^p\sigma_l(X_t)dW^l_t\right),\end{equation}
where the exponential map is due to the connection. Therefore, we get the following statement.
\begin{lemma}
The Intrinsic SDE \[\left(V,\{\sigma_1, ..., \sigma_p\}, G\right)\] has an equivalent Belopolskya-Daletskii form that is given by
\begin{equation}dX_t = exp_{X_t}\left(V(X_t)dt + \dfrac{1}{2}\sum_{l = 1}^p \nabla^G_I(\sigma_l) dt + \sum_{l = 1}^p\sigma_l(X_t)dW^l_t\right).\end{equation}
\end{lemma}
This lemma allows us to take advantage of the underlying exponential map for numerical computations. e.g. a first order numerical method can be given by,
\begin{equation}
\label{equation: BD form _iteration}
X_{t+\delta t} = exp_{X_t}(Y_{t+\delta t} - Y_t),
\end{equation}
where $Y_t$ is a stochastic process in the tangent space $T_{X_t}M$ such that,
\begin{equation}
\label{equation: BD form in tangent space.}
dY_s = \left[V(X_t) + \dfrac{1}{2}\sum_{l = 1}^p \nabla^G_I(\sigma_l(X_t))\right]ds + \sum_{l=1}^p\sigma_l(X_t) dW^l_s.
\end{equation}
Instead of converting Intrinsic SDE into Belopolskya-Daletskii form, one may also choose to convert the Intrinsic SDE into a Stratonovich SDE and use numerical methods from \cite{malham2008stochastic}. In literature one finds many numerical methods for Stratonovich SDEs on manifolds, e.g. in \cite{malham2008stochastic,castell1995efficient}. Alternatively, the option of numerical computations in local chart is always available.
\subsection{Extended Ito formulae on manifolds.}
\label{section:extended_ito}
Let us consider a function $F:\mathbb{R}\times \mathbb{R}^n\to \mathbb{R}^m$ such that for constant $x$
\[\delta F(t,x) = V(F(t,x)) dt + \sum_{l=1}^p \sigma_l(F(t,x))\circ dW^l_t.\]
If $\delta X_t = a(t)dt + \sum_{l=1}^p B_l(t) \circ dW^l_t$, then
\[\delta F(t,X_t) = \left[ V(F(t,X_t)) + D_2F(t,X_t)a(t) \right] dt \]
\begin{equation}
+ \sum_{l=1}^p\left[ \sigma_l(F(t,X_t))+  D_2F(t,X_t)B_l(t)\right] \circ dW^l_t.
\end{equation}
Instead of Stratonovich representation, if the SDEs are given in Ito form then the \textit{extended Ito formula} gives the Ito SDE for $F(t,X_t)$. This formula is also known as \textit{generalized Ito formula} or \textit{Ito-Wentzell's formula}. Readers can refer \cite{kunita1981some, KUNITA1984249, kunita2006decomposition} for further details on the formula. As per this formula, if
\[dF(t,x) = V(F(t,x)) dt + \sum_{l=1}^p \sigma_l(F(t,x)) dW^l_t,\]
and \[dX_t = a(t)dt + \sum_{l=1}^p B_l(t) dW^l_t,\]
then 
\begin{multline}
d(F(t,X_t)) =  \left[V(F(t,X_t)) + D_2F(t,X_t) a(t) \right] dt\\ + \left[\dfrac{1}{2}\sum_{l=1}^p B^T_l(t)D^2_2F(t,X_t) B_l(t)\right] dt \\+\left[ \dfrac{1}{2}\sum_{l=1}^p B^T_l(t)D_2F(t,X_t) D\sigma_l(F(t,X_t))\right] dt \\+ \left[\sum_{l=1}^p \sigma_l(F(t,X_t))+ D_2F(t,X_t)B_l(t)\right] dW^l_t.
\end{multline}
In this section we give an equivalent formula for Intrinsic SDEs on manifolds. For this, let us consider $F:\mathbb{R}\times M\to N$,
such that $F(t,x)$ is a semi-martingale for every $x\in M$,
If $X_t$ is a semi-martingale on $M$, then the following proposition gives the generalized formula for the semi-martingale $F(t,X_t)$.

\begin{proposition}
\label{proposition:gen_ito_formula}
Let $F:\mathbb{R}\times M\to N$,
such that for constant $x$,
\begin{equation}
\mathbf{d}(F(t,x)) = \left[V(F(t,x)) + \dfrac{1}{2} \sum_{l = 1}^p {}^N G(\sigma_l (F(t,x)))\right]dt + \sum_{l=1}^p \sigma_l (F(t,x)) dW^l_t,
\end{equation}
where the manifolds $M$ and $N$ are equipped with diffusion generators ${}^M G$ and ${}^N G$ respectively. Let $X_t$ be a semi-martingale on $M$, with intrinsic representation as
\[\mathbf{d}X_t = \left[a(t) + \dfrac{1}{2} \sum_{l = 1}^p {}^MG(B_l(t))\right]dt + \sum_{l=1}^p B_l(t) dW^l_t.\]
Then,
\begin{multline}
\label{equation: extended_Ito_intrinsic}
\mathbf{d}(F(t,X_t))= \left[ V(F(t,X_t)) + D_2F(t,X_t)a(t) + \dfrac{1}{2} \sum_{l = 1}^p \nabla^{{}^N G}_S(\sigma_l (F(t,X_t))) \right]  dt\\+ \dfrac{1}{2} \sum_{l = 1}^p {{}^N G_S}(\sigma_l (F(t,X_t)) +D_2F(t,X_t)B_l(t))dt\\
+ \dfrac{1}{2} \sum_{l = 1}^p D_2F(t,X_t) \nabla^{{}^M G}_S(B_l(t)) dt +\sum_{l=1}^p \left[ \sigma_l (F(t,X_t)) + D_2F(t,X_t)B_l(t)\right]dW^l_t.
\end{multline}
\end{proposition}
\begin{proof}
In Stratonovich representation 
\[\delta (F(t,x)) = \left[V(F(t,x)) + \dfrac{1}{2} \sum_{l = 1}^p \nabla^{{}^N G}_S(\sigma_l (F(t,x)))\right]dt + \sum_{l=1}^p \sigma_l (F(t,x)) \circ dW^l_t,\]
and 
\[\delta X_t = \left[a(t) + \dfrac{1}{2} \sum_{l = 1}^p \nabla^{{}^M G}_S(B_l(t))\right]dt + \sum_{l=1}^p B_l(t)\circ dW^l_t.\]
Therefore,
\[\delta F(t,X_t) = \left[V(F(t,x)) + \dfrac{1}{2} \sum_{l = 1}^p \nabla^{{}^N G}_S(\sigma_l (F(t,x)))\right]_{x = X_t} dt\]
\[+ \sum_{l=1}^p \sigma_l (F(t,x))\big\vert_{x = X_t} \circ dW^l_t + D_2F(t,X_t) \circ \delta X_t\]
\[= \left[ V(F(t,X_t)) + D_2F(t,X_t)a(t) + \dfrac{1}{2} \sum_{l = 1}^p \nabla^{{}^N G}_S(\sigma_l (F(t,X_t)))\right] dt\] \[+ \dfrac{1}{2} \sum_{l = 1}^pD_2F(t,X_t) \nabla^{{}^M G}_S(B_l(t)) dt+\sum_{l=1}^p \left[ \sigma_l (F(t,X_t)) + D_2F(t,X_t)B_l(t)\right] \circ dW^l_t.\]
Converting it back into the Intrinsic representation, 
\[\mathbf{d}(F(t,X_t))= \left[ V(F(t,X_t)) + D_2F(t,X_t)a(t) + \dfrac{1}{2} \sum_{l = 1}^p \nabla^{{}^N G}_S(\sigma_l (F(t,X_t)))\right] dt\]\[+ \dfrac{1}{2} \sum_{l = 1}^p {{}^N G_S}(\sigma_l(F(t,X_t)) +D_2F(t,X_t)B_l(t)) dt\]
\[ + \dfrac{1}{2} \sum_{l = 1}^pD_2F(t,X_t) \nabla^{{}^M G}_S(B_l(t)) dt +\sum_{l=1}^p \left[ \sigma_l (F(t,X_t)) + D_2F(t,X_t)B_l(t)\right]dW^l_t.\]
\end{proof}
Equation \eqref{equation: extended_Ito_intrinsic} is the extended Ito formula on manifolds when the semi-martingale $X_t\in M$ is in the intrinsic SDE representation. If $X_t\in M$ is given as a Stratonovich SDE, then the extended Ito formula on manifolds is given by equation \eqref{equation:extended_Ito_stratonovich} in the following statement.
\begin{corollary_p}
Let $F:\mathbb{R}\times M\to N$,
such that for constant $x$,
\begin{equation}
\mathbf{d}(F(t,x)) = \left[V(F(t,x)) + \dfrac{1}{2} \sum_{l = 1}^p {}^N G(\sigma_l (F(t,x)))\right]dt + \sum_{l=1}^p \sigma_l (F(t,x)) dW^l_t,
\end{equation}
where the manifolds $N$ is equipped with diffusion generator ${}^N G$. Let $X_t$ be a semi-martingale on $M$, with Stratonovich representation as
\[\delta X_t = a(t)dt + \sum_{l=1}^p B_l(t)\circ dW^l_t.\]
Then,
\begin{multline}
\label{equation:extended_Ito_stratonovich}
\mathbf{d}(F(t,X_t))= \left[ V(F(t,X_t)) + D_2F(t,X_t)a(t) + \dfrac{1}{2} \sum_{l = 1}^p \nabla^{{}^N G}_S(\sigma_l (F(t,X_t))) \right]  dt\\+ \dfrac{1}{2} \sum_{l = 1}^p {{}^N G_S}(\sigma_l (F(t,X_t)) +D_2F(t,X_t)B_l(t))dt+\sum_{l=1}^p \left[ \sigma_l (F(t,X_t)) + D_2F(t,X_t)B_l(t)\right]dW^l_t.
\end{multline}
\end{corollary_p}
\begin{proof}
The intrinsic representation of $X_t$ is obtained using the Stratonovich diffusion generator ${}^MG_S$. In proposition \ref{proposition:gen_ito_formula} if we consider ${}^MG_S$ as the diffusion generator on $M$, then $\nabla^{{}^M G}_S = {}^MG_S - {}^MG_S = 0$.
\end{proof}
\section{Concluding remarks.}
We have shown that the SDE obtained using the diffusion generator constructed using the first order differential equation is nothing but the Intrinsic representation of the Stratonovich SDE. Therefore, in case of Stratonovich SDE, i.e. in the case of SDE obtained using the diffusion generator constructed using the first order differential equation, we do not require additional information about the state space $M$. On the other hand, for the second order differential equation, we need an additional structure on the state space $M$. If this information is available in terms of a connection, then the resulting SDE is the traditional Ito SDE on manifolds (we have called this the standard Ito SDE). We have shown that this additional structure/information on $M$ can also be provided using a regular Lagrangian. We have considered the Hamiltonian dynamics on the tangent bundle $TM$ to construct the diffusion generator. On the Riemannian manifold, if the Lagrangian is kinetic energy then the resulting SDE is, again, the Intrinsic representation of the standard Ito SDE. On the basis of the way of construction of the diffusion generator, we can classify the SDEs on manifolds as shown in the following flowchart.
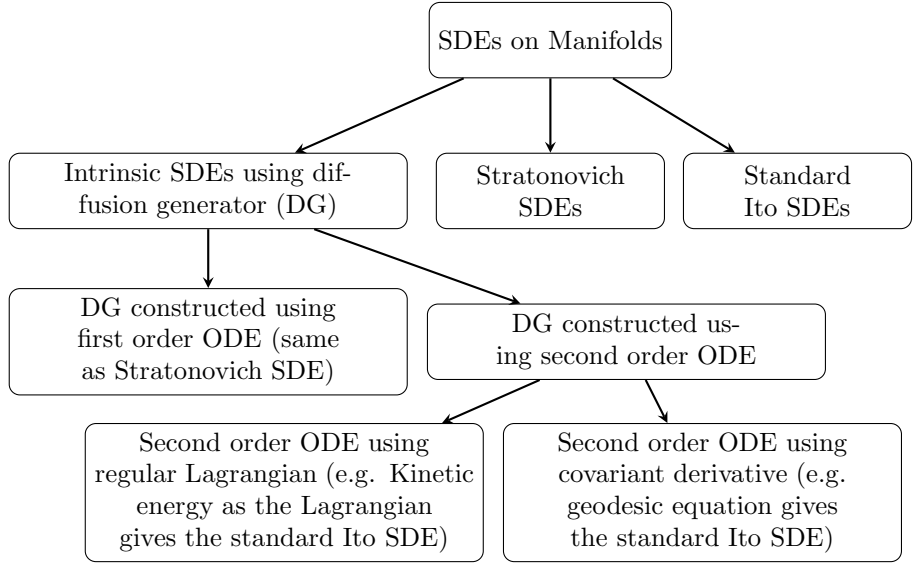
\begin{center}
\begin{tikzpicture}[node distance=2cm]
\node (start) [startstop] {{SDEs on Manifolds}};
\node (start1) [process, below of=start, xshift=-4.5cm] {{Intrinsic SDEs using diffusion generator (DG)}};
\node (base2) [starts, below of=start, xshift=0cm] {{Stratonovich SDEs}};
\node (base3) [starts, below of=start, xshift=+3.25cm] {{Standard Ito SDEs}};
\node (pro1) [process, below of=start1, xshift=0cm] {{DG constructed using first order ODE} (same as {Stratonovich SDE})};
\node (pro2) [process, below of=start1, xshift=+5.5cm] {{DG constructed using second order ODE}};
\node (pro3) [process, below of=pro2, xshift=-4.5cm] {Second order ODE using {regular Lagrangian} (e.g. Kinetic energy as the Lagrangian gives the {standard Ito SDE})};
\node (pro4) [process, below of=pro2, xshift=+1cm] {Second order ODE using {covariant derivative} (e.g. geodesic equation gives the {standard Ito SDE})};
\draw [arrow] (start) -- (start1);
\draw [arrow] (start) -- (base2);
\draw [arrow] (start) -- (base3);
\draw [arrow] (start1) -- (pro1);
\draw [arrow] (start1) -- (pro2);
\draw [arrow] (pro2) -- (pro3);
\draw [arrow] (pro2) -- (pro4);
\end{tikzpicture}
\end{center}

The advantage of the diffusion generator approach is that it makes the co-ordinate invariant analysis of SDE on manifolds easier. This was demonstrated in section \ref{section:last}, wherein we have derived a formulae to convert a given Intrinsic SDE into an SDE obtained using a different diffusion generator. Using the conversion formulae, we have derived extended/generalized Ito formulae on manifolds.

\appendix
\section{Proof that equation \eqref{equation: SDE_on mani} is co-ordinate invariant.}
\label{appendix: proof of co-ordinate invariance.}
Suppose for vector field $\sigma_l\in \mathfrak{X}(M)$, locally in chart $(U,\chi)$ with coordinates $(x^1, x^2, ..., x^n)$, the diffusor $\alpha_l = G(\sigma_l)$ is given as
$\tilde{\alpha_l} = G(\sigma_l)|_U = a^i_l\dfrac{\partial}{\partial x^i} + \sigma^i_l \sigma^j_l \dfrac{\partial}{\partial x^i \partial x^j}$. In chart $(U,\chi)$, the left hand side of equation \eqref{equation: SDE_on mani} is given by
\begin{equation}
\mathbf{d}X_t|_U = dX^i_t\dfrac{\partial}{\partial x^i} +\dfrac{1}{2} d[X^i_t,X^j_t]\dfrac{\partial^2}{\partial x^i \partial x^j},
\end{equation}
where $X^i_t = \chi^i(X_t)$. Therefore, in chart $(U,\chi)$, we get the Ito SDEs,
\begin{equation}
dX^i_t = (V^i+ \dfrac{1}{2}\sum_{l = 1}^p a^i_l)dt + \sigma^i_l dW^l_t
\end{equation}
and
\begin{equation}
d[X^i_t,X^j_t] = \sigma^i_l(X_t)\sigma^j_l(X_t) dt.
\end{equation}
Similarly, in chart $(U,\Upsilon)$ with coordinates $(y^1, ..., y^n)$, the Ito SDE is given by
\begin{equation}
d\breve{X}^i_t = (\breve{V}^i+ \dfrac{1}{2}\sum_{l = 1}^p \breve{a}^i_l)dt + \breve{\sigma}^i_l dW^l_t,
\end{equation}
where $\breve{X}^i_t = \Upsilon^i(X_t)$. Moreover, using the change of coordinates formula, we know that $\breve{\alpha}_l= \mathfrak{D}\Upsilon \tilde{\alpha_l}$. Therefore, it can be concluded that 
\begin{equation}
\label{eqaution: drift_vector field transition in diffusion}
\breve{a}_l^i = \dfrac{\partial\Upsilon^i}{\partial x^k} a^k_l + \sigma_l^k\sigma^m_l\dfrac{\partial \Upsilon^i}{\partial x^k \partial x^m}.
\end{equation}

Let the transition map from chart $(U,\chi)$ to $(U,\Upsilon)$ be given by $\Psi = \Upsilon\circ \chi^{-1}: \mathbb{R}^n\to \mathbb{R}^n$. Let the coordinates in the co-domain of the chart $(U,\chi)$ be given by $(\tilde{x}^1, \tilde{x}^2, ..., \tilde{x}^n)$ and let $\tilde{X}_t = (X_t^1, X_t^2, ..., X_t^n)$ i.e., $\tilde{X}_t = \chi(X_t)$.
By Ito's lemma,
\[d(\Psi^i(\tilde{X_t})) = \dfrac{\partial \Psi^i}{\partial \tilde{x}_j}\left( V^j + \dfrac{1}{2}\sum_{l = 1}^p \left(a^j_l(X_t)\right)dt + \sigma^j_l(X_t) dW^l_t\right)\]
\begin{equation}
+ \dfrac{1}{2}\dfrac{\partial^2 \Psi^i}{\partial \tilde{x}^j \partial \tilde{x}^k}\sigma^j_l(X_t)\sigma^k_l(X_t) dt.
\end{equation}
As $\Psi^i(\tilde{X_t}) = \Upsilon^i(X_t)$, if $\breve{X}^i_t = \Upsilon^i(X_t)$, we can rewrite the above equation as
\begin{equation}
\label{equation:Ito_transition_SDE}
d \breve{X}^i_t = \dfrac{\partial \Psi^i}{\partial \tilde{x}_j}\left( V^i + \dfrac{1}{2}\sum_{l = 1}^p \left(a^j_l(X_t)\right)dt + \sigma^j_l(X_t) dW^l_t\right)+ \dfrac{1}{2}\dfrac{\partial^2 \Psi^i}{\partial \tilde{x}^j \partial \tilde{x}^k}\sigma^j_l(X_t)\sigma^k_l(X_t) dt.
\end{equation}
But we know that in chart $(U,\Upsilon)$, the Ito SDE representation for $\breve{X}^i_t = \Upsilon^i(X_t)$ is given by
\begin{equation}
\label{equation:Ito_SDE_in another chart}
d\breve{X}^i_t = \left(\breve{V}^i + \dfrac{1}{2}\sum_{l = 1}^p \breve{a}^i_l(X_t)\right)dt + \breve{\sigma^i_l}(X_t) dW^l_t,
\end{equation}
using equation \eqref{eqaution: drift_vector field transition in diffusion},
\begin{subequations}
\begin{equation}d\breve{X}^i_t = \left(\breve{V}^i + \dfrac{1}{2}\sum_{l = 1}^p\dfrac{\partial\Upsilon^i}{\partial x^k} a^k_l +\dfrac{1}{2}\sum_{l = 1}^p \sigma_l^k\sigma^m_l\dfrac{\partial \Upsilon^i}{\partial x^k \partial x^m}\right)dt + \breve{\sigma^i_l}(X_t) dW^l_t\end{equation}
\begin{equation}= \left(\dfrac{\partial\Upsilon^i}{\partial x^j}V^j + \dfrac{1}{2}\sum_{l = 1}^p\dfrac{\partial\Upsilon^i}{\partial x^k} a^k_l +\dfrac{1}{2}\sum_{l = 1}^p \sigma_l^k\sigma^m_l\dfrac{\partial \Upsilon^i}{\partial x^k \partial x^m}\right)dt + \dfrac{\partial\Upsilon^i}{\partial x^j}{\sigma^j_l}(X_t) dW^l_t.\end{equation}
\end{subequations}
As it is known that $\dfrac{\partial \Upsilon^i}{\partial x^j} = \dfrac{\partial \Psi^i}{\partial \tilde{x}^j}$ and $\dfrac{\partial^2 \Upsilon^i}{\partial x^m \partial x^j} = \dfrac{\partial^2 \Psi^i}{\partial \tilde{x}^m \partial \tilde{x}^j}$, equation \eqref{equation:Ito_transition_SDE} and equation \eqref{equation:Ito_SDE_in another chart} are equivalent. Therefore, equation \eqref{equation: SDE_on mani} is coordinate invariant.

\bibliography{2}

\begin{thebibliography}{10}

\bibitem{abraham2008foundations}
Ralph Abraham and Jerrold~E Marsden.
\newblock {\em Foundations of mechanics}.
\newblock Number 364. American Mathematical Soc., 2008.

\bibitem{armstrong2018intrinsic}
John Armstrong and Damiano Brigo.
\newblock Intrinsic stochastic differential equations as jets.
\newblock {\em Proceedings of the Royal Society A: Mathematical, Physical and
  Engineering Sciences}, 474(2210):20170559, 2018.

\bibitem{arnold1974stochastic}
Ludwig Arnold.
\newblock {\em Stochastic Differential Equations: Theory and Applications}.
\newblock Wiley–Blackwell, 1974.

\bibitem{castell1995efficient}
Fabienne Castell and Jessica Gaines.
\newblock An efficient approximation method for stochastic differential
  equations by means of the exponential lie series.
\newblock {\em Mathematics and computers in simulation}, 38(1-3):13--19, 1995.

\bibitem{constantin2008stochastic}
Peter Constantin and Gautam Iyer.
\newblock A stochastic lagrangian representation of the three-dimensional
  incompressible navier-stokes equations.
\newblock {\em Communications on Pure and Applied Mathematics: A Journal Issued
  by the Courant Institute of Mathematical Sciences}, 61(3):330--345, 2008.

\bibitem{holm2020Ito_wentzell}
Aythami~Bethencourt de~Leon, Darryl~D Holm, Erwin Luesink, and So~Takao.
\newblock Implications of kunita--it{\^o}--wentzell formula for k-forms in
  stochastic fluid dynamics.
\newblock {\em Journal of Nonlinear Science}, 30:1421--1454, 2020.

\bibitem{elworthy1982stochastic}
Kenneth~David Elworthy.
\newblock {\em Stochastic differential equations on manifolds}, volume~70.
\newblock Cambridge University Press, 1982.

\bibitem{emery2006two}
Michel {\'E}mery.
\newblock On two transfer principles in stochastic differential geometry.
\newblock In {\em S{\'e}minaire de Probabilit{\'e}s XXIV 1988/89}, pages
  407--441. Springer, 2006.

\bibitem{emery2012stochastic}
Michel {\'E}mery.
\newblock {\em Stochastic calculus in manifolds}.
\newblock Springer Science \& Business Media, 2012.

\bibitem{gliklikh2011global}
Yuri~E Gliklikh.
\newblock {\em Global and stochastic analysis with applications to mathematical
  physics}.
\newblock Springer, 2011.

\bibitem{hsu2002stochastic}
Elton~P Hsu.
\newblock {\em Stochastic analysis on manifolds}.
\newblock Number~38. American Mathematical Soc., 2002.

\bibitem{kunita1981some}
Hiroshi Kunita.
\newblock Some extensions of ito's formula.
\newblock In {\em S{\'e}minaire de Probabilit{\'e}s XV 1979/80}, pages
  118--141. Springer, 1981.

\bibitem{KUNITA1984249}
Hiroshi Kunita.
\newblock First order stochastic partial differential equations.
\newblock In Kiyosi Itô, editor, {\em Stochastic Analysis}, volume~32 of {\em
  North-Holland Mathematical Library}, pages 249--269. Elsevier, 1984.

\bibitem{kunita2006decomposition}
Hiroshi Kunita.
\newblock On the decomposition of solutions of stochastic differential
  equations.
\newblock In {\em Stochastic Integrals: Proceedings of the LMS Durham
  Symposium, July 7--17, 1980}, pages 213--255. Springer, 2006.

\bibitem{malham2008stochastic}
Simon~JA Malham and Anke Wiese.
\newblock Stochastic lie group integrators.
\newblock {\em SIAM Journal on Scientific Computing}, 30(2):597--617, 2008.

\bibitem{meyer1981geometrie}
Paul-Andr{\'e} Meyer.
\newblock G{\'e}om{\'e}trie stochastique sans larmes.
\newblock In {\em S{\'e}minaire de Probabilit{\'e}s XV 1979/80: Avec table
  g{\'e}n{\'e}rale des expos{\'e}s de 1966/67 {\`a} 1978/79}, pages 44--102.
  Springer, 1981.

\bibitem{oksendal2013stochastic}
Bernt Oksendal.
\newblock {\em Stochastic differential equations: an introduction with
  applications}.
\newblock Springer Science \& Business Media, 2013.

\bibitem{rossi2022rough}
Emilio Rossi~Ferrucci.
\newblock Rough path perspectives on the ito-stratonovich dilemma.
\newblock \url{https://spiral.imperial.ac.uk/handle/10044/1/96036}, 2022.
\newblock [Online, last accessed: Nov 2022].

\bibitem{schwartz1982geometrie}
Laurent Schwartz.
\newblock Geometrie differentielle du 2 {\`e}me ordre, semi-martingales et
  equations differentielles stochastiques sur une variete differentielle.
\newblock 1982.

\end{thebibliography}

\end{document}